\documentclass[a4paper,english,nolineno]{socg-lipics-v2019}

\usepackage{microtype}
\usepackage{amsthm}
\usepackage{amsmath}
\usepackage{mathbbol}
\usepackage{amsfonts}
\usepackage{amssymb}
\usepackage{graphicx}
\usepackage{algorithm}
\usepackage{url}

\usepackage{algorithm}
\usepackage{algpseudocode}
\usepackage{comment}
\usepackage[small,nohug,heads=vee]{diagrams}
\diagramstyle[labelstyle=\scriptstyle]

\usepackage{todonotes}

\newcommand{\N}{\mathbb{N}}

\newcommand{\R}{\mathbb{R}}

\newcommand{\F}{\mathbb{F}}

\newcommand{\eps}{\epsilon}
\newcommand{\id}{\mathrm{id}}

\newcommand{\ignore}[1]{}

\renewcommand{\dim}{\mathrm{dim}}
\renewcommand{\ker}{\mathrm{ker}}
\DeclareMathOperator{\Ima}{Im}
\newcommand{\coker}{\mathrm{coker}}

\title{Chunk Reduction for Multi-Parameter Persistent Homology}

\titlerunning{Chunk Reduction for Multi-Parameter Persistent Homology}

\author{Ulderico Fugacci}{Graz University of Technology, Graz, Austria}{fugacci@tugraz.at}{https://orcid.org/0000-0003-3062-997X}{}

\author{Michael Kerber}{Graz University of Technology, Graz, Austria}{kerber@tugraz.at}{https://orcid.org/0000-0002-8030-9299}{}

\authorrunning{U. Fugacci and M. Kerber}

\Copyright{Ulderico Fugacci and Michael Kerber}

\ccsdesc[500]{Theory of computation~Randomness, geometry and discrete structures}
\ccsdesc[500]{Computing methodologies~Shared memory algorithms}

\keywords{Multi-parameter persistent homology, Matrix reduction, Chain complexes}

\category{}

\relatedversion{}

\supplement{}

\funding{Supported by the Austrian Science Fund (FWF) grant number P 29984-N35.}

\acknowledgements{We thank Sara Scaramuccia for initial discussions on this project, Wojciech Chacholski, Michael Lesnick and Francesco Vaccarino for helpful suggestions,
and Federico Iuricich for his help on the experimental comparison
with~\cite{Iuricich2018}.
The datasets used in the experimental evaluation are courtesy of the AIM@SHAPE data repository \cite{aimatshape2006}.}

\EventEditors{Gill Barequet and Yusu Wang}
\EventNoEds{2}
\EventLongTitle{35th International Symposium on Computational Geometry (SoCG 2019)}
\EventShortTitle{SoCG 2019}
\EventAcronym{SoCG}
\EventYear{2019}
\EventDate{June 18--21, 2019}
\EventLocation{Portland, United~States}
\EventLogo{socg-logo}
\SeriesVolume{129}
\ArticleNo{37} 

\hideLIPIcs

\begin{document}

\maketitle

\begin{abstract}
The extension of persistent homology to multi-parameter setups is an algorithmic challenge. Since most computation tasks scale badly with the size of the input complex, an important pre-processing step consists of simplifying the input while maintaining the homological information. We present an algorithm that drastically reduces the size of an input. Our approach is an extension of the chunk algorithm for persistent homology (Bauer et al., Topological Methods in Data Analysis and Visualization III, 2014).
We show that our construction produces the smallest multi-filtered chain complex among all the complexes quasi-isomorphic to the input, improving on the guarantees of previous work in the context of discrete Morse theory.
Our algorithm also offers an immediate parallelization scheme in shared memory. Already its sequential version compares favorably with existing simplification schemes, as we show by experimental evaluation.
\end{abstract}

\section{Introduction}\label{sec:intro}
In the last decades, topology-based tools are gaining a more and more relevant role in the analysis and in the extraction of the core information of unorganized, high-dimensional and potentially large datasets.
Thanks to its capability of keeping track of the changes in the homological features of a dataset which evolves with respect to a parameter, {\em persistent homology} has represented a real game-changer in this field. Recently, an extension of persistent homology called {\em multi-parameter persistent homology} is drawing the attention of a growing number of researchers.
In a nutshell, multi-parameter persistence generalizes the classic persistent homology by studying multivariate datasets which are filtered by two or more (independent) scale parameters.
Multi-parameter persistent homology of a dataset cannot be captured by complete discrete invariant~\cite{Carlsson2009}, but this has not prevented the researchers from defining several descriptors based on multi-parameter persistence \cite{Cerri2013, Knudson2008, Miller2017}.
Since these descriptors tend to have high algorithmic complexity,
it is a natural pre-processing step to pass to a smaller
but equivalent representation of the input complex
and to invoke any demanding computation on this smaller representation.

\subparagraph{Contribution.}
Inspired by the chunk algorithm \cite{Bauer2014} for persistent homology, we present a reduction algorithm which for a filtered dataset returns a filtered chain complex having the same multi-parameter persistence but a drastically smaller size.
Our approach is based on the observation that even in the absence
of a global persistence diagram, local matrix reductions yield
pairs of simplices which can be eliminated from the boundary matrix
without affecting the homological information.
Our approach proceeds in two steps, first identifying such local pairs
of simplices, and in a second step manipulating the non-local columns
of the boundary matrix such that all indices of locally paired simplices
disappear. Both steps permit a parallel implementation
with shared memory.
We prove that our simplification scheme is optimal in the sense that
any chain complex quasi-isomorphic to the input complex
must contain at least as many generators as our output complex.
Our algorithm yields similar time and space complexity bounds
as its one-parameter counterpart.
We implemented our algorithm, making use of various techniques
that have proven effective for persistence computation, such as
the twist reduction~\cite{Chen2011} and efficient data structure
for the columns of a boundary matrix~\cite{Bauer2017}.
We experimentally show that our implementation is effective (see next
paragraph).
For the sake of clarity, our approach is described for the two-parameter case. No constraint prevents the generalization of the proposed method to an arbitrary number of parameters.

\subparagraph{Comparison with related work.}
Our work is motivated by a line of research on complex simplification
using discrete Morse theory (DMT)~\cite{Iuricich2016, Allili2017, Allili2018, Scaramuccia2018}.
In these works, the idea is to build a discrete gradient locally and to
return the resulting Morse complex on critical simplices as a simplification.
In analogy to the one-parameter case, our chunk algorithm is
an attempt to realize this simplification scheme using persistent homology
instead of DMT. This gives more flexibility, as in DMT, the paired cells
are constrained to be incident, while this restriction is not present
for persistence pairs. Consequently, we are able to prove optimality
of our output size in general, while DMT-based approaches only succeeded
to give guarantees for special cases such as multi-filtrations of 3D regular grids and of abstract simplicial 2-complexes \cite{ScaramucciaPhD}.
In practice, the theoretical benefit in terms of output size is rather
marginal, as our experiments show; however, the timings show that
our improved theoretical guarantee comes without performance
penalty; on the contrary, our algorithm is always faster than the DMT-based approach
presented in \cite{Scaramuccia2018} on all tested examples.
We remark, however,
that the DMT-based algorithm returns a complex endowed with a Forman gradient
as well as the corresponding discrete Morse complex,
which can be of potential use for other application domains
than computing persistent homology.

A related question is the computation of a \emph{minimal presentation}
of a persistence module induced by a simplicial or general chain complex.
Roughly speaking, a presentation consists of a finite set
of (graded) generators and relations, and the minimal presentation
is one with the minimal possible number of generators and relations.
Our algorithm does not yield
a minimal presentation, however, since more computations are needed
to identify generators of the chain complex as generators or as relations.
An algorithm by Lesnick and Wright%
\footnote{\url{https://www.ima.umn.edu/2017-2018/SW8.13-15.18/27428}}
\cite{lw-unpublished} computes such a minimal
presentation through matrix reduction in cubic time,
carefully choosing a column
order to reduce the number of reduction instances.
Their algorithm is used in their software package RIVET~\cite{Lesnick2015}.
Our contribution can serve as a pre-processing step for their algorithm,
reducing the size of the matrix through efficient and possibly parallelized
computation and invoking their global reduction step on a much smaller
chain complex.

\section{Background}\label{sec:back}

\subparagraph{Homology.}
Fixed a base field $\F$ and a positive integer $d$, let us consider, for each $0\leq k \leq d$, a finite collection of elements denoted as {\em $k$-generators} (or, equivalently, generators of dimension $k$). A finitely generated chain complex $C_*=(C_k, \partial_k)$ over $\F$ is a collection of pairs $(C_k, \partial_k)$ where:
\begin{itemize}
  \item $C_k$ is the $\F$-vector space spanned by the generators of dimension $k$,
  \item $\partial_k: C_k \rightarrow C_{k-1}$ is called {\em boundary} map and it satisfies the property that $\partial_{k-1}\partial_k=0$.
\end{itemize}
The elements of $C_k$ are called {\em $k$-chains} and, by definition, are $\F$-linear combinations of the generators of dimension $k$. The {\em support} of a $k$-chain is the set of $k$-generators whose coefficient in the chain is not zero.
In the following, we will simply use the term chain complex in place of finitely generated chain complex. Moreover, we will assume that, given a chain complex, an explicit set of generators is provided.
\medskip\\
Given a chain complex $C_*$, we denote as $Z_k:=\ker \,\partial_k$ the space of the $k$-cycles of $C_*$, and as $B_k:=\Ima \, \partial_{k+1}$ the space of the $k$-boundaries of $C_*$. The {\em $k^{th}$ homology space} of $C_*$ is defined as the vector space $H_k(C_*):=Z_k/B_k$.
The rank $\beta_k$ of the $k^{th}$ homology space of a chain complex $C_*$ is called the \textit{$k^{th}$ Betti number} of $C_*$.
\medskip\\
A \emph{chain map} $f_*:C_*\to D_*$ is a collection of linear maps $f_k:C_k\to D_k$ which commutes with the boundary operators of $C_*$ and of $D_*$.
 A simple example of a chain map is the inclusion map, if $C_*$ is a subcomplex
of $D_*$. In general, a chain map $f_*$ induces linear maps $H_k(C_*)\to H_k(D_*)$ for every $k$.
\medskip\\
Chain complexes allow for capturing the combinatorial and the topological structure of a discretized topological space.
Given a finite simplicial complex $K$, the chain complex $C_*$ associated to $K$ is defined by setting $C_k$ as the $\F$-vector space generated by the $k$-simplices of $K$ and the boundary $\partial_k(c)$ of a $k$-chain $c$ corresponding to a $k$-simplex $\sigma$ as the collection of the $(k-1)$-simplices lying on the geometrical boundary of $\sigma$.
Consequently, homology of a finite simplicial complex $K$ is defined as the homology of the chain complex $C_*$ associated to $K$. Intuitively, homology spaces of $K$ reveal the presence of ``holes'' in the simplicial complex.
The non-null elements of each homology space are cycles, which do not represent the boundary of any collection of simplices of $K$.
Specifically, $\beta_0$ counts the number of connected components of $K$, $\beta_1$ its tunnels and holes, and $\beta_2$ the shells surrounding voids or cavities.

\subparagraph{Multi-parameter persistent homology.}
In the following, we will focus for simplicity on datasets filtered by two independent scale parameters. All definitions and results in this paper
can be generalized to more parameters without problems.
Let $p=(p_x, p_y), q=(q_x, q_y) \in \R^2$, we will write throughout $p\leq q$ if $p_x \leq q_x$ and $p_y \leq q_y$.
Given a chain complex $C_*=(C_k, \partial_k)$, let us assume to have an assignment which, for each generator $g \in C_k$, returns a value $v(g)\in \R^2$ such that, for any generator $g' \in C_{k-1}$, if $\langle \partial_k g, g'\rangle \neq 0$, then $v(g')\leq v(g)$.
The assignment $v$ can be extended to a function $v:C_*\rightarrow \R^2$ assigning to each chain $c$ the least common upper bound $v(c)$ of the values $v(g)$ of the generators in the support of $c$. In the following, $v(c)$ will be called the {\em value} of $c$.
\medskip\\
Given a chain complex $C_*=(C_k, \partial_k)$ endowed with a value function $v:C_*\rightarrow \R^2$ and fixed a value $p\in \R^2$ we define $C_*^p=(C_k^p, \partial_k^p)$ as the chain complex for which:
\begin{itemize}
  \item $C_k^p$ is the space of the $k$-chains of $C_k$ having value lower than or equal to $p$,
  \item $\partial_k^p$ is the restriction of $\partial_k$ to $C_k^p$.
\end{itemize}
By the definition of $v$, for any chain $c$ of dimension $k$, we have that $v(\partial_k(c))\leq v(c)$. So, $C_*^p$ is well-defined.
For $p, q \in \R^2$ with $p\leq q$, $C_*^p$ is a chain subcomplex of $C_*^q$. For this reason, we denote the collection $C$ of the chain complexes $C_*^p$ with $p \in \R^2$ as {\em bifiltered chain complex}.
\medskip\\
Given $p, q \in \R^2$ with $p \leq q$, the inclusion map from $C_*^p$ to $C_*^q$ induces a linear map between the corresponding homology spaces $H_k(C_*^p)$ and $H_k(C_*^p)$.
The {\em multi-parameter persistence $k^{th}$ module} $H_k(C)$ of a bifiltered chain complex $C$ is the collection of the homology spaces $H_k(C_*^p)$ with $p$ varying in $\R^2$ along with all the linear maps induced by the inclusion maps.
\medskip\\
Let $C, D$ be two bifiltered chain complexes. We call $C$ and $D$ {\em homology-equivalent}
if, for any fixed $k\in\mathbb{N}$, $H_k(C_*^p)$ and $H_k(D_*^p)$ are isomorphic via a map $\psi^p_k$ and, for any $p, q \in \R^2$ with $p\leq q$, the diagram
\begin{diagram}
H_k(C_*^p)         &\rTo_{\qquad}   &H_k(C_*^q)\\
\dTo_{\psi^p_k \,}  &           &\dTo_{\psi^q_k \,}\\
H_k(D_*^p)         &\rTo_{\qquad}   &H_k(D_*^q)
\end{diagram}
commutes where horizontal maps are induced by inclusion maps.
Moreover, we call $C$ and $D$ {\em quasi-isomorphic} if the isomorphisms of the above diagram are induced by a collection of chain maps $f^p_*: C^p_* \rightarrow D^p_*$ satisfying, for any $p\leq q$ and any $k$, the commutative diagram
\begin{diagram}
C^{p}_k         &\rTo^{\qquad}   &C^q_k\\
\dTo_{f^p_k \,}  &           &\dTo_{f^q_k \,}\\
D^{p}_k          &\rTo^{\qquad}   &D^{q}_k
\end{diagram}
in which horizontal maps are the inclusion maps.
By definition, quasi-isomorphic bifiltered chain complexes are homology-equivalent. The converse, as depicted in Figure \ref{fig:counterexample}, does not hold in general.
\begin{figure}[!htb]
	\centering
	\begin{tabular}{ccccc}
    \includegraphics[width=.35 \linewidth]{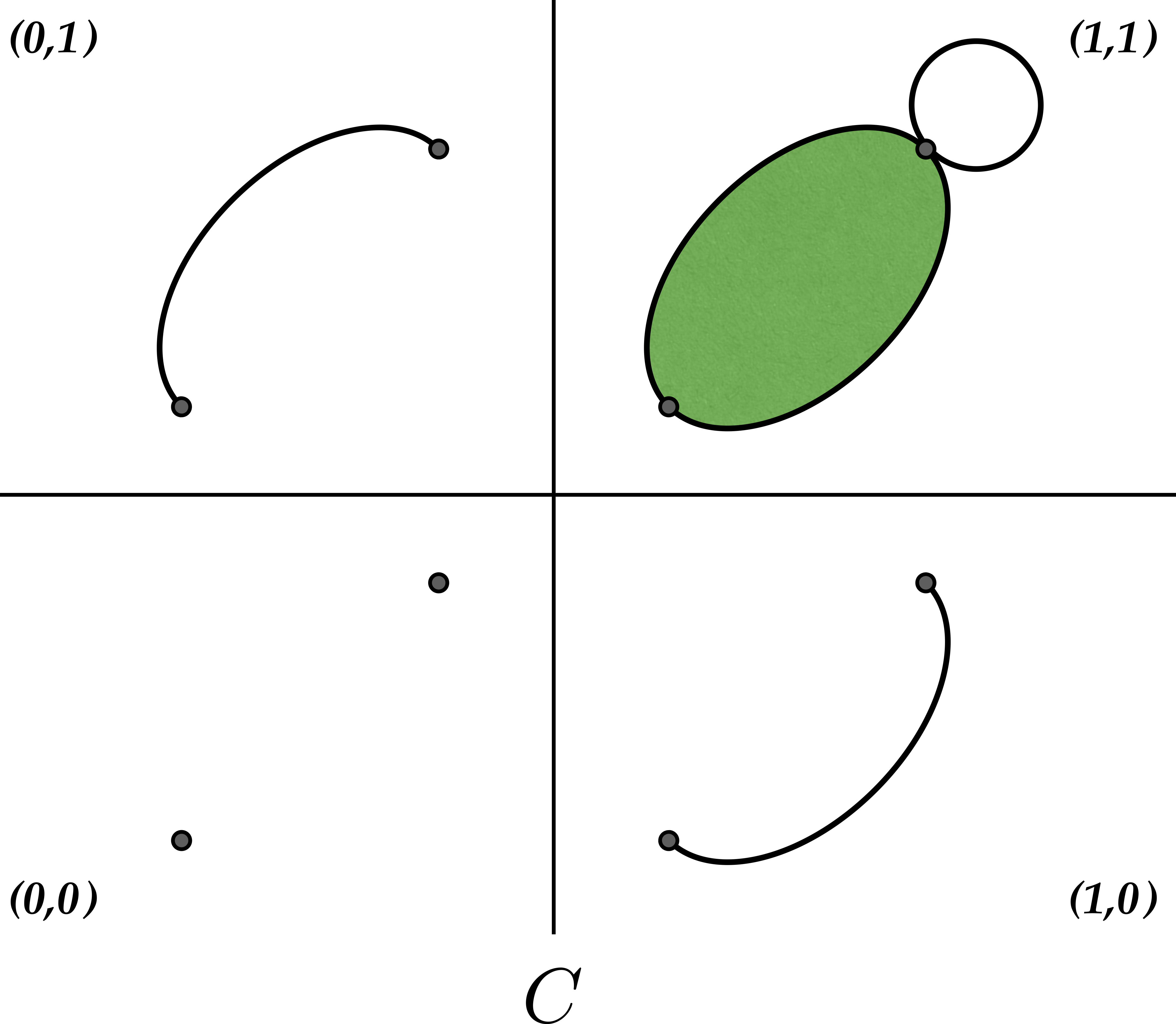} & & & &
    \includegraphics[width=.35 \linewidth]{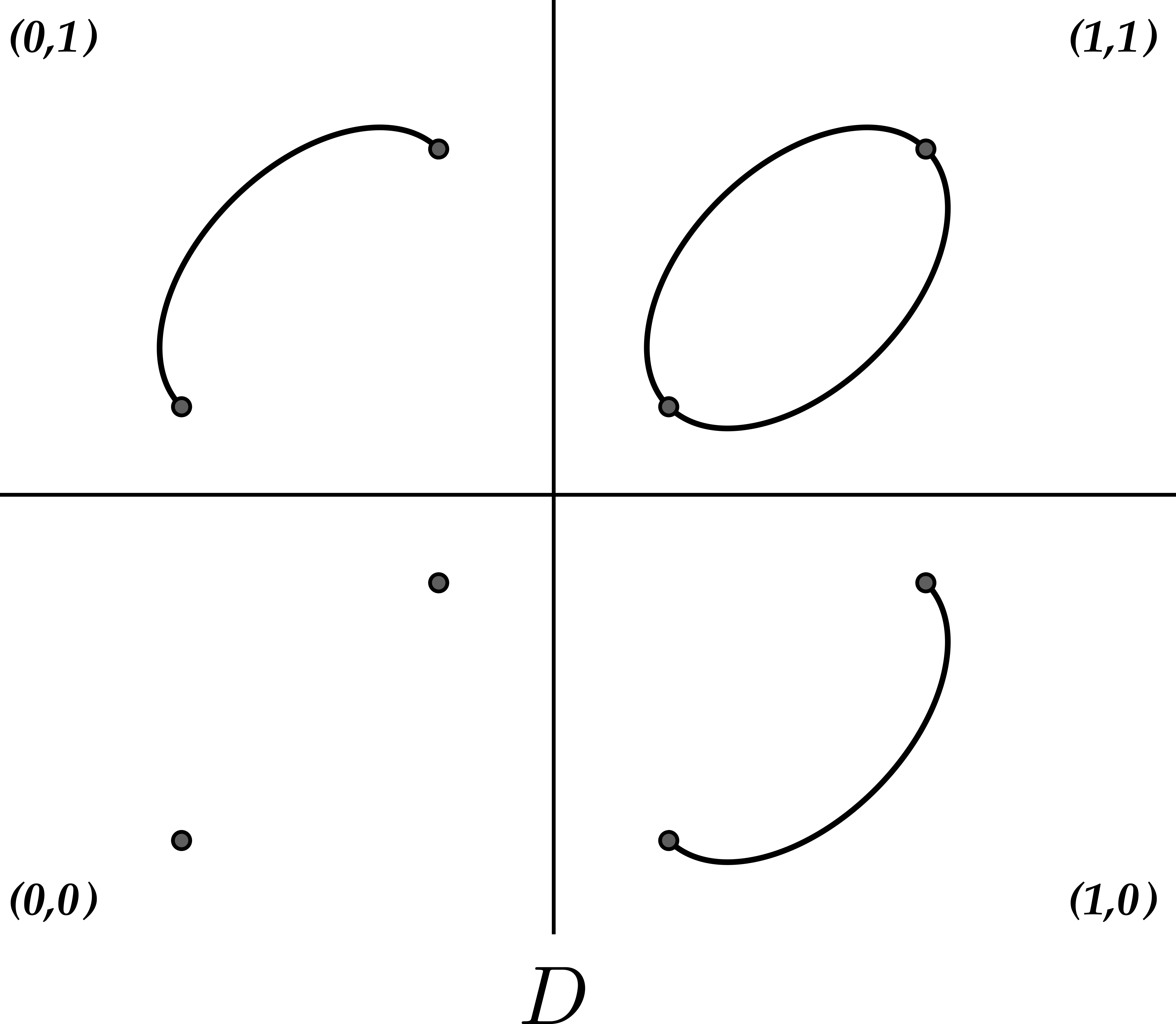}\\
	\end{tabular}
	\caption{The two depicted bifiltered chain complexes are homology-equivalent but not quasi-isomorphic.
}
	\label{fig:counterexample}
\end{figure}
Two chain maps $\alpha_*, \beta_*: C_* \rightarrow D_*$ are called {\em chain-homotopic} if there exists a collection of maps $\phi_k: C_k \rightarrow D_{k+1}$ such that, for any $k$, $\partial^D_{k+1}\phi_k + \phi_{k-1} \partial^C_{k} = \alpha_k - \beta_k$.
Two bifiltered chain complexes $C$ and $D$ are called {\em homotopy-equivalent} if, there exist two collections of chain maps $f^p_*: C^p_* \rightarrow D^p_*$ and  $g^p_*: D^p_* \rightarrow C^p_*$ such that, for any $p \in \R^2$, $f^p_*$ and $g^p_*$ are homotopy-inverse one with respect to the other (i.e., $g^p_* f^p_*$ and $f^p_*g^p_*$ are chain-homotopic to $\id_{C^p_*}$ and $\id_{D^p_*}$, respectively) and they satisfy, for any $p\leq q$, the following commutative diagram
\begin{diagram}
C^{p}_k         &\rTo^{\qquad}   &C^q_k\\
\dTo^{f^p_k \,} \uTo_{g^p_k \,}  &           &\dTo^{f^q_k \,} \uTo_{g^q_k \,}\\
D^{p}_k          &\rTo^{\qquad}   &D^{q}_k
\end{diagram}

in which horizontal maps are the inclusion maps.
Homotopy-equivalent bifiltered chain complexes are necessarily quasi-isomorphic.

\subparagraph{Representation of bifiltered chain complexes.}
We assume that the bifiltered chain complex in input is provided as a finite sequence of triples of the form
\[(p, k, \Delta )\]
where $p\in\R^2$, $k$ is a positive integer, and $\Delta$ is a finite list of pairs $(l, \lambda)$ with $\lambda \in \F$ and $l\in \N$.
In order to retrieve the bifiltered chain complex $C$ represented by this list, each triple $(p, k, \Delta)$ has to be interpreted as a $k$-generator $c$ with value $v(c)=p$ and whose boundary is encoded through $\Delta$. Specifically, by defining as $c_l$ the generator stored in position $l$ of the sequence of triples, the pair $(l, \lambda)$ in $\Delta$ represents that $c_l$ appears in the boundary of $c$ with coefficient $\lambda$.
In order to ensure that the above-described interpretation of the input sequence actually returns a valid bifiltered chain complex, we require that, if $c'$ is a generator appearing in the boundary of $c$, then the value $v$ associated to $c'$ is lower than or equal to the one taken by $c$.

Given a chain complex $C_*$ endowed with a value function $v$, let us consider an injective function $i$ providing the set of generators of $C_*$ with a total order which is consistent with the partial order induced by the value funtion $v$. More precisely, given two generators $c_1, c_2$ of $C_*$, the {\em index} function $i$ has to satisfy that $v(c_1)\leq v(c_2)$ implies $i(c_1)\leq i(c_2)$.

We uniquely represent the bifiltered chain complex $C$ as follows.
Each generator $c$ of $C_*$ is stored in a data structure that we call the \emph{column} of the generator.
A column stores:
\begin{itemize}
  \item the index $i(c)$,
  \item the value $v(c)$,
  \item the dimension $k$ of $g$,
  \item the boundary of $c$.
\end{itemize}
The boundary is stored as a (possibly empty) container of entries of the form $(i(c'),\lambda)$, where $c'$ is a generator of dimension $k-1$ such that $\lambda:=\langle\partial_k(c), c'\rangle\neq 0$.
In such a case, we say that $c'$ is in the boundary of $c$.
The boundary of $c$ is then the linear combination induced by this container.
The name ``column'' comes from the idea that the collection of all columns
can be visualized as a matrix.
If we decorate the columns of the matrix with extra information (index, value, and dimension),
we obtain the data structure from above.
We will also write $k$-column if the column represents a generator of dimension $k$.
Due to the fact that a column is nothing but an encoding of a generator, in the following, with a small abuse of notation we will often use interchangeably the two terms.
Given a $k$-column $c$, we define the {\em local pivot} of $c$ as the generator in the boundary of $c$ with maximal index such that $v(c')=v(c)$. If no such
generator exists, we simply say that $c$ has no local pivot.
\medskip\\
Arguably, the most common case of input data is simplicial complexes.
Let $K$ denote a finite simplicial complex. A \emph{bifiltration}
is a collection $(K^p)_{p\in\R^2}$ of subcomplexes of $K$
such that $K^p\subseteq K^q$ whenever $p\leq q$.
Fixing a simplex $\sigma\in K$, we say that $p\in\R^2$ is \emph{critical}
for $\sigma$, if $\sigma\in K^{p}$, but $\sigma\notin K^{p-(\eps,0)}\cup K^{p-(0,\eps)}$ for every $\eps>0$. The bifiltration is called \emph{$h$-critical}
with $h\geq 1$, if for every $\sigma\in K$, there are at most $h$ critical
positions in $\R^2$. For instance, $1$-critical means that every simplex
in $K$ enters the filtration at a unique minimal value.
Such 1-critical bifiltrations can easily be described by a sequence of the
form $(p,k,\Delta)$ as above by adding one line per simplex which defines
its critical position, its dimensions, and its boundary.
The case of $h$-critical bifiltrations can be handled with the following
trick (see also~\cite{Chacholski2017}). Letting $p_1,\ldots,p_h$ denote the critical positions of $\sigma$, sorted
by $x$-coordinate, we introduce $h$ distinct $k$-generators $c_1,\ldots,c_h$
of the form
$(p_i,k,\Delta)$, that is, $h$ copies of the same simplex.
For two consecutive positions $p_i=(x_i,y_i)$ and $p_{i+1}=(x_{i+1},y_{i+1})$,
we add an additional $(k+1)$-generator at $(x_{i+1},y_i)$ (which is the smallest
point $q$ such that $p_i\leq q$ and $p_{i+1}\leq q$), whose boundary is
equal to $c_i-c_{i+1}$.

\section{Algorithm}\label{sec:alg}

Given a bifiltered chain complex $C$ encoded as a collection of columns,
we define the \emph{chunk algorithm} that returns as output
a bifiltered chain complex as a collection of columns
that is homotopy-equivalent to $C$ and has fewer generators.
The algorithm works in three phases:
\begin{itemize}
  \item {\em local reduction};
  \item {\em compression};
  \item {\em removal of local columns}.
\end{itemize}

\subparagraph{Phase I: Local reduction.}
The goal of this phase is to label columns as global,
local positive or local negative columns.
Initially, all columns are unlabeled.
We proceed in decreasing dimensions, from the top-dimensions
of $C$ down to $0$.
For dimension $k$,
the algorithm traverses the $k$-columns in increasing order
with respect to $i$ and performs the following operations on a $k$-column $c$.
If $c$ is already labeled, do nothing.
Otherwise,
as long as $c$ has a local pivot and there is a $k$-column $c'$ with $i(c')<i(c)$
and the same local pivot as $c$,
perform the column addition $c\gets c+\lambda c'$,
where $\lambda$ is chosen such that the local pivot of $c$ disappers.
If at the end of this loop, the column $c$ does not have a local pivot,
we label the column as global and proceed.
Otherwise, we label $c$ as local negative and its local pivot $c'$ as local positive.
We call $(c',c)$ a \emph{local pair} in this case.
This ends the description of Phase I of the algorithm.

Note that within the local reduction, any column addition
of the form $c\gets c+\lambda c'$ implies that $v(c)=v(c')$.
Hence, the local reduction operates independently on columns
of the same value. We call blocks of columns with the same value
\emph{chunks}; hence the name of the algorithm.
Operations on one chunk do not affect columns on any other chunk,
hence the local reduction phase can be readily invoked in parallel
on the chunks of the chain complex.

Finally, note that by proceeding in decreasing dimension, we avoid
performing any column additions on local positive columns.
That is reminiscent of the \emph{clearing optimization}
in the one-parameter version~\cite{Bauer2014, Chen2011}.

\subparagraph{Phase II: Compression.}
In the second phase, the algorithm removes local (positive or negative)
generators from the boundary of all global columns in the matrix:

For each global $k$-column $c$, while the boundary of the column contains a generator that is local positive or local negative, the algorithm picks the
local $(k-1)$-generator $c'$ with maximal index.
\begin{itemize}
  \item if $c'$ is negative, remove $c'$ from the boundary of $c$;
  \item if $c'$ is positive, denote $c''$ as the (unique) local negative $k$-column with $c'$ as local pivot and perform the column addition $c\gets c+\lambda c''$, where $\lambda$ is chosen such that $c'$ disappears in the boundary of $c$.
\end{itemize}
This ends the description of the compression phase.
On termination, all columns in the boundary of a global $k$-column are global $(k-1)$-columns.

The above process terminates for a column $c$ because the index of the maximal
local generator in the boundary of $c$ is strictly decreasing in each step.
That is clear for the case that $c'$ is local negative. If $c'$ is local positive,
then $c'$ is the generator in the boundary of $c''$ with the maximal index, so the column
addition does not introduce in the boundary of $c$ any generators with a larger index.

Note that the compression of a global column does not affect the result
on any other global column. Thus, the phase can be parallelized as well.

\subparagraph{Phase III: Removal of local pairs.}
In this step, the chain complex becomes smaller. The procedure is simple:
traverse all columns, and remove all columns labeled
as local (positive or negative). Return the remaining (global) columns
as resulting chain complex. This finishes the description of the phase and the entire chunk algorithm.

\section{Correctness}\label{sec:correct}

In this section, we prove that the presented algorithm returns a bifiltered chain complex that is homotopy-equivalent to the input.
For that, we define two elementary operations on chain complexes:

\begin{description}
\item [Order-preserving column addition.] An operation of the form
$c\gets c+\lambda c'$ is called \emph{order-preserving}
if $v(c')\leq v(c)$. Note that such an operation maintains the property
that any generator $c''$ in the boundary of $c$
satisfies $v(c'')\leq v(c)$,
by transitivity of $v$. We remark that order-preserving column additions
are the generalization of left-to-right column additions in the one-parameter case.

\item [Removal of a local pair.]
Fix a local pair $(c_1,c_2)$, that means, $c_1$ is a $k$-column,
$c_2$ is a $(k+1)$-column, $v(c_1)=v(c_2)$ and $c_1$ is the local pivot of $c_2$.
We call {\em removal of the local pair} $(c_1, c_2)$ the operation $Del(c_{1},c_{2})$ which acts on the columns as follows.
\begin{itemize}
\item For every $(k+1)$-column $c$, replace its boundary $\partial_{k+1}(c)$ with $\partial_{k+1}(c) - \lambda^{-1} \mu \partial_{k+1}(c_2)$, where $\lambda$ and $\mu$ are the coefficients of $c_1$ in $\partial_{k+1}(c_2)$ and in $\partial_{k+1}(c)$, respectively. In particular, after this operation, $c_1$ disappeared from the boundary of any $(k+1)$-columns.
\item For every $(k+2)$-column $c$, update its boundary by setting the coefficient of $c_{2}$ in $\partial_{k+2}(c)$ to $0$. Visualizing the chain complex
as a matrix, this corresponds to removing the row corresponding to $c_2$.
\item Delete the columns $c_1$ and $c_{2}$.
\end{itemize}
Note that all column additions performed in the first step are
order-preserving because the pair $(c_1,c_2)$ is local,
that is, $v(c_1)=v(c_2)$.
\end{description}

We will show at the end of this section that both elementary operations
leave the homotopy type the bifiltered chain complex unchanged.

\begin{theorem}
Let $\bar{C}$ denote the bifiltered chain complex computed by the chunk algorithm
from the previous section on an input bifiltered chain complex $C$.
Then, $\bar{C}$ and $C$ are homotopy-equivalent.
\end{theorem}
\begin{proof}
The idea is to express the chunk algorithm by a sequence of order-preserving
column additions and removals of local pairs.
Because every column addition in Phase I is between columns of the same
value, all column additions are order-preserving.
Hence, after Phase I, the chain complex is equivalent to the input.

In Phase II, note that all column additions performed are order-preserving.
Indeed, if $c'$ is in the boundary of column $c$, then $v(c')\leq v(c)$ holds.
If $c'$ is local positive, it triggers a column addition
of the form $c\gets c+\lambda c''$ with
its local negative counterpart $c''$. Since $v(c')=v(c'')$,
$v(c'')\leq v(c)$ as well.

A further manipulation in Phase II is the removal of local negative columns from the boundary of global columns.
These removals cannot be directly expressed in terms of the two elementary operations
from above. Instead, we define a slight variation of our algorithm:
in Phase II, when we encounter a local negative $c'$, we do nothing.
In other words, the compression only removes the local positive generators
from the boundary $c$, and keeps local negative and global generators.
In Phase III, instead of removing local columns, we perform a removal
of a local pair $(c_1,c_2)$ whenever we encounter a local negative column $c_2$
with local pivot $c_1$. We call that algorithm the \emph{modified chunk algorithm}.
Note that this modified algorithm is a sequence of order-preserving
column additions, followed by a sequence of local pair removals, and thus
produces a chain complex that is equivalent to the input $C$.

We argue next that the chunk algorithm and the modified chunk algorithm
yield the same output. Since both versions eventually remove all local
columns, it suffices to show that they yield the same global columns.
Fix an index of a global column, and let $c$ denote the column of that index
returned by the original chunk algorithm. Let $c^\ast$ denote the
column of the same index produced by the modified algorithm after
the modified Phase II. The difference of $c^\ast$ and $c$ lies in the
presence of local negative generators in the boundary of $c^\ast$ which have been removed in $c$.
The modified Phase III affects $c^\ast$ in the following way:
when a local pair $(c_1,c_2)$ is removed, the local negative
$c_2$ is, if it is present, removed from the boundary of $c^\ast$. There is no column addition
during the modified Phase III involving $c^\ast$ because all local positive
columns have been eliminated. Hence, the effect of the modified Phase III
on $c^\ast$ is that all local negative columns are removed from its boundary which turns
$c^\ast$ to be equal to $c$ at the end of the algorithm.
Hence, the output of both algorithms is the same, proving the theorem.
\end{proof}

We proceed with the proofs that both elementary operations
yield homotopy-equivalent bifiltered chain complexes.

\subparagraph{Order-preserving column addition.}
Given two $k$-columns $c_1, c_2$ such that $v(c_1)\leq v(c_2)$ and a scalar value $\lambda \in \F$, the algorithm we propose allows for adding $\lambda$ copies of the boundary of column $c_1$ to the boundary of column $c_2$. In this section, we formalize that in terms of modifications of chain complexes and we prove that this operation does not affect multi-parameter persistent homology.
\medskip\\
Given a chain complex $C_*=(C_l, \partial_l)$ endowed with a value function $v:C_*\rightarrow \R^2$, let us consider two generators $c_1, c_2$ among the ones of the space of the $k$-chains $C_k$ such that $\langle c_1, c_2 \rangle=0$ and $v(c_1)\leq v(c_2)$. Chosen a scalar value $\lambda \in \F$, let us define $\bar{C}_*=(\bar{C}_l, \bar{\partial}_l)$ by setting:
\begin{itemize}
  \item $\bar{C}_l=C_l$,
  \item for any $c\in \bar{C}_l$, \[\bar{\partial}_l(c)=\begin{cases}
  \partial_k(c) + \lambda \langle c, c_2 \rangle \partial_k (c_1) & \text{ if } l=k,\\
  \partial_{k+1}(c) - \lambda \langle \partial_{k+1} (c), c_2 \rangle c_1 & \text{ if } l=k+1,\\
  \partial_l(c) & \text{ otherwise.}
  \end{cases}\]
\end{itemize}

As formally proven in Appendix \ref{sec:app_details_1}, $\bar{C}_*$ is a chain complex.
\medskip\\
Let us define the maps $f_*: C_* \rightarrow \bar{C}_*$, $g_*: \bar{C}_* \rightarrow C_*$ as follows:
\begin{itemize}
  \item for any $c\in C_l$, \[f_l(c)=\begin{cases}
  c - \lambda \langle c, c_2 \rangle c_1 & \text{ if } l=k,\\
  c & \text{ otherwise;}
  \end{cases}\]
  \item for any $c\in \bar{C}_l$, \[g_l(c)=\begin{cases}
  c + \lambda \langle c, c_2 \rangle c_1 & \text{ if } l=k,\\
  c & \text{ otherwise.}
  \end{cases}\]
\end{itemize}

The just defined maps enable to prove the following result (see Appendix \ref{sec:app_details_1} for detailed proofs).

\begin{proposition}
  $C_*$ and $\bar{C}_*$ are isomorphic via the chain map $f_*$ and its inverse $g_*$.
\end{proposition}

Since the function $v$ is a valid value function for $\bar{C}_*$ inducing a bifiltered chain complex $\bar{C}$ and, for any $p\in \R^2$, the restrictions $f^p_l:C^p_l \rightarrow \bar{C}^p_l$ and $g^p_l:\bar{C}^p_l \rightarrow C^p_l$ of the maps $f_l$ and $g_l$, respectively, are well-defined (see Appendix \ref{sec:app_details_1}), we are ready to prove the equivalence between the two bifiltered chain complexes $C$ and $\bar{C}$.

\begin{theorem}\label{thm:column_add}
  $C$ and $\bar{C}$ are homotopy-equivalent.
\end{theorem}

\begin{proof}
  The previous results enables us to claim that, for any $p,q \in \R^2$ with $p\leq q$, the following diagram (in which horizontal maps are the inclusion maps)
  \begin{diagram}
  C^p_l         &\rTo_{\qquad}   &C^q_l\\
  \dTo_{f^p_l \,}  &           &\dTo_{f^q_l \,}\\
  \bar{C}^p_l         &\rTo_{\qquad}   &\bar{C}^q_l
  \end{diagram}
  commutes and the maps $f^p_l$, $f^q_l$ are isomorphisms.
\end{proof}

\subparagraph{Removal of a local pair.}
Given a local pair $(c_1, c_2)$ of columns, the algorithm we propose allows for deleting them from the boundary matrix. In this section, we formalize that in terms of modifications of chain complexes and we prove that this operation does not affect multi-parameter persistent homology.
\medskip\\
Given a chain complex $C_*=(C_l, \partial_l)$ endowed with a value function $v:C_*\rightarrow \R^2$, let us consider two generators $c_1, c_2$ among the ones of the space of the $k$-chains $C_k$ and of the space of the $(k+1)$-chains $C_{k+1}$, respectively, such that $\lambda:=\langle \partial_{k+1}(c_2), c_1 \rangle \neq 0$ and $v(c_1)=v(c_2)$.
Let us define $\bar{C}_*=(\bar{C}_l, \bar{\partial}_l)$ by setting:
\begin{itemize}
  \item the space of the $l$-chains $\bar{C}_l$ as \[\bar{C}_l=\begin{cases}
  \{c\in C_k \, |\, \langle c, c_1 \rangle = 0 \} & \text{ if } l=k,\\
  \{c\in C_{k+1} \, |\, \langle c, c_2 \rangle = 0 \} & \text{ if } l=k+1,\\
  C_l & \text{ otherwise;}
  \end{cases}\]
  \item for any $c\in \bar{C}_l$, \[\bar{\partial}_l(c)=\begin{cases}
  \partial_{k+1}(c) - \lambda^{-1} \langle \partial_{k+1}(c), c_1 \rangle \partial_{k+1}(c_2) & \text{ if } l=k+1,\\
  \partial_{k+2}(c) - \langle \partial_{k+2} (c), c_2 \rangle c_2 & \text{ if } l=k+2,\\
  \partial_l(c) & \text{ otherwise.}
  \end{cases}\]
\end{itemize}

As formally proven in Appendix \ref{sec:app_details_1}, $\bar{C}_*$ is a chain complex.
\medskip\\
Let us define the maps $r_*: C_* \rightarrow \bar{C}_*$, $s_*: \bar{C}_* \rightarrow C_*$ as follows:
\begin{itemize}
  \item for any $c\in C_l$, \[r_l(c)=\begin{cases}
  c - \lambda^{-1} \langle c, c_1 \rangle \partial_{k+1}(c_2) & \text{ if } l=k,\\
  c - \langle c, c_2 \rangle c_2 & \text{ if } l=k+1,\\
  c & \text{ otherwise;}
  \end{cases}\]
  \item for any $c\in \bar{C}_l$, \[s_l(c)=\begin{cases}
  c - \lambda^{-1} \langle \partial_{k+1}(c), c_1 \rangle c_2 & \text{ if } l=k+1,\\
  c & \text{ otherwise;}
  \end{cases}\]
\end{itemize}

The just defined maps enable to prove the following result (see Appendix \ref{sec:app_details_1} for detailed proofs).

\begin{proposition}
  The maps $r_*$ and $s_*$ are chain maps which are homotopy-inverse one with respect to the other.
\end{proposition}

This latter combined with the fact that the function $v$ is a valid value function for $\bar{C}_*$ inducing a bifiltered chain complex $\bar{C}$ and, for any $p\in \R^2$, the restrictions $r^p_l:C^p_l \rightarrow \bar{C}^p_l$, $s^p_l:\bar{C}^p_l \rightarrow C^p_l$ of the maps $r_l$ and $s_l$, as well as the restrictions of the maps ensuring that $r_*$ and $s_*$ are homotopy-inverse, are well-defined (see Appendix \ref{sec:app_details_1}), guarantees the homotopy-equivalence between the two bifiltered chain complexes $C$ and $\bar{C}$.

\begin{theorem}
  $C$ and $\bar{C}$ are homotopy-equivalent.
\end{theorem}

\section{Optimality and complexity}\label{sec:optimal}

\subparagraph{Optimality.}
Let $D$ be a bifiltered chain complex.
For $p\in\R^2$, we define
\[D^{<p}_*:=\sum_{q<p} D^q_*,\]
where the sum of chain complexes, analogously to the sum of the vector spaces, is the chain complex spanned by the union of the generators of the summands.
Moreover, let $\eta^p_k$ be the homology map in dimension $k$ induced by the inclusion of $D^{<p}_*$ into $D^{p}_*$.
We denote the number of variations in the $k^{th}$ homology space occurred at value $p$ as
\[\delta^p_k(D):=\dim\,\ker \, \eta^p_{k-1} + \dim\,\coker \, \eta^p_k,\]
and the number of $k$-generators added at value $p$ in $D$ as
\[\gamma^p_k(D):=\dim\,D^p_k - \dim\,D^{<p}_k.\]

\begin{theorem}\label{thm:optimality}
  Let $\bar{C}$ be the bifiltered chain complex obtained by applying the chunk algorithm to the bifiltered chain complex $C$. We have that $\delta^p_k(C)=\gamma^p_k(\bar{C})$.
\end{theorem}

The full proof is given in Appendix \ref{sec:app_details_2}
and can be summarized as
follows. Every global column with value $p$
either destroys a homology class of $H(C^{<p})$,
or it creates a new homology class in $H(C^p)$, which is not destroyed
by any other column of value $p$. Hence, each global column contributes
a generator to $\ker \, \iota^p_{k-1}$ or to $\coker \, \iota^p_k$, where $\iota^p_l$ is the map between the $l^{th}$ homology spaces induced by the inclusion of $C^{<p}_*$ into $C^{p}_*$.
Local columns do not contribute to either of these two spaces.
The result follows from the fact that the number of global columns
at value $p$ is precisely the number of generators added at $\bar{C}$.

The next statement shows that our construction is optimal
in the sense that any bifiltered chain complex $D$ that is quasi-isomorphic to $C$ must have at least as many generators as $\bar{C}$.

\begin{theorem}\label{lemma:leq}
Any bifiltered chain complex $D$ quasi-isomorphic to $C$ has to add at least $\delta^p_k(C)$ $k$-generators at value $p$. I.e., $\delta^p_k(C)\leq \gamma^p_k(D)$.
\end{theorem}

The full proof is also given in Appendix \ref{sec:app_details_2}.
To summarize it,
it is not too hard to see that, for any bifiltered chain complex $D$,
\[\delta^p_k(D)\leq \gamma^p_k(D)\]
holds. Moreover, the quasi-isomorphism of $C$ and $D$ implies that $\dim\,\ker \, \eta^p_{k-1}=\dim\,\ker \, \iota^p_{k-1}$
and $\dim\,\coker \, \eta^p_k=\dim\,\coker \, \iota^p_k$. So,
\[\delta^p_k(C)= \delta^p_k(D)\]
which implies that claim. The equality of the dimensions
is formally verified using the Mayer-Vietoris sequence and the $5$-lemma
to establish an isomorphism from $H_k(C^{<p}_*)$ to $H_k(D^{<p}_*)$
that commutes with the isomorphism at value $p$.

\subparagraph{Complexity.}
In order to properly express the time and the space complexity of the proposed algorithm, let us introduce the following parameters. Given a bifiltered chain complex $C$, we denote as $n$ the number of generators of $C$, as $m$ the number of chunks (i.e., the number of different values assumed by $v$), as $\ell$ the maximal size of a chunk, and as $g$ the number of global columns. Moreover, we assume the maximal size of the support of the boundary of the generators of $C$ as a constant. The latter condition is always ensured for bifiltered simplicial complex of fixed dimension.

\begin{theorem}
The chunk algorithm has time complexity $O(m\ell^3\log \ell+g\ell n\log\ell)$ and space complexity $O(n\ell + g^2)$.
\end{theorem}
\begin{proof}
Due to the similarity between the two algorithms, the analysis of complexity of the proposed algorithm is analogous to the first two steps  of the one-parameter chunk algorithm \cite{Bauer2014}.
The additional factor of $\log\ell$ comes from our choice of using priority queues as column type and could be removed by using list representations as in~\cite{Bauer2014}.%
\footnote{We remark, however, that this would result in a performance penalty in practice. See~\cite{Bauer2017}.}

On the space complexity, during the Phase I, the generators in the boundary of any column can be at most $O(\ell)$. So, $O(n\ell)$ is a bound on the accumulated size of all columns
after Phase I.
During Phase II, the boundary of any global column can consist of up to $n$ generators, but reduces to $g$ generators at the end of the compression of the column
because all local entries have been removed. Hence, the final chain complex has at most $g$ entries in each of its $g$ columns, and requires $O(g^2)$ space.
Hence, the total space complexity is $O(n\ell+n+g^2)$, where the second summand is redundant.\footnote{We remark that this bound only holds for the sequential version of the
algorithm. In a parallelized version, it can happen that several compressed columns achieve a size of $O(n)$ at the same time.}
\end{proof}

\section{Implementation and experimental results}\label{sec:impl}

We briefly introduce the developed implementation of the chunk algorithm and we evaluate its performance.
The current C++ implementation of the chunk algorithm consists of approximately 350 lines of code. It takes as input a finite bifiltered chain complex expressed accordingly to the representation described in Section \ref{sec:back} and encodes each column as a  \texttt{std::priority\_queue}. Even if the chunk algorithm permits a parallelization in shared memory, this first implementation does not exploit this potential. Our experiments have been performed on a configuration Intel Xeon CPU E5-1650 v3 at 3.50 GHz with 64 GB of RAM.

For testing the implemented chunk algorithm, we have compared its performances with the simplification process based on discrete Morse theory proposed in \cite{Scaramuccia2018} (DMT-based algorithm) whose implementation is publicly available \cite{Iuricich2018}.
We remark once more that the DMT-based approach yields a somewhat richer
output than solely a simplified chain complex with the same homotopy;
however, we only compare the size of the resulting structure
(in terms of the number of generators per chain group) in this
experimental comparison.

In our experiments, we have considered both synthetic and real datasets represented as simplicial complexes. The latter ones are courtesy of the AIM@SHAPE data repository \cite{aimatshape2006}.
Most of the datasets are of dimension 2 or 3 and are embedded in a 3D environment. For our experiments, we have adopted as filtering functions the ones obtained by extending to all the simplices of the complex the $x$ and the $y$ coordinates assigned to the vertices of the datasets.
Table \ref{table:experiments} displays the achieved results.

The data sets are given in off file format, that means,
as a list of triangles specified by boundary vertices.
In order to apply our algorithm, we first have to convert the data
into a boundary matrix representation. Hence, we enlist
in Table~\ref{table:experiments} the \emph{preparation time} to create the boundary matrix and the \emph{simplification time} to perform our chunk algorithm. In contrast, the DMT-based approach
avoids the initial construction of the boundary matrix but transforms
the input into a different data structure before starting its simplification
step. We also list the running times of these separate steps in
Table~\ref{table:experiments}.
In both cases, we do not list the time for
reading the input file into memory and writing the output structure on disk.

\begin{table}[!htb]
\resizebox{\columnwidth}{!}{%
\centering
\begin{tabular}{l | c c | c c c c | c c }
	\multirow{3}{*}{Dataset} &
  \multicolumn{2}{c|}{Size} &
	\multicolumn{4}{c|}{Time (sec.)} &
  \multicolumn{2}{c}{Memory Usage (GB)} \\
	& \multirow{2}{*}{Input} & \multirow{2}{*}{Output} & \multicolumn{2}{c|}{Chunk} & \multicolumn{2}{c|}{DMT} & \multicolumn{1}{c|}{\multirow{2}{*}{\, Chunk \,}} & \multirow{2}{*}{DMT} \\
& & & Prep. & \multicolumn{1}{c|}{Simpl.} & Prep. & Simpl. & \multicolumn{1}{c|}{ } & \\
   \hline
  Eros & 2.9 M & 202 K & 1.7 & \multicolumn{1}{c|}{0.8} & 2.7 & 15.8 & \multicolumn{1}{c|}{0.36} & 0.46 \\
  Donna & 3.0 M & 217 K & 1.8 & \multicolumn{1}{c|}{0.8} & 2.8 & 16.9 & \multicolumn{1}{c|}{0.38} & 0.48 \\
  Chinese Dragon & 3.9 M & 321 K & 2.5 & \multicolumn{1}{c|}{1.1} & 3.9 & 22.3 & \multicolumn{1}{c|}{0.52} & 0.64 \\
  Circular Box & 4.2 M & 365 K & 2.9 & \multicolumn{1}{c|}{1.2} & 4.3 & 24.0 & \multicolumn{1}{c|}{0.68} & 0.68 \\
  Ramesses & 5.0 M & 407 K & 3.4 & \multicolumn{1}{c|}{1.3} & 5.5 & 29.4 & \multicolumn{1}{c|}{0.68} & 0.81 \\
  Pensatore & 6.0 M & 369 K & 3.8 & \multicolumn{1}{c|}{1.6} & 6.8 & 34.3 & \multicolumn{1}{c|}{0.76} & 0.97 \\
  Raptor & 6.0 M & 260 K & 4.4 & \multicolumn{1}{c|}{1.7} & 5.4 & 32.3 & \multicolumn{1}{c|}{0.73} & 0.93 \\
  Neptune & 12.0 M & 893 K & 8.4 & \multicolumn{1}{c|}{4.4} & 14.9 & 69.2 & \multicolumn{1}{c|}{1.52} & 1.94 \\
  Cube 1 & 590 K & 67 K & 0.5 & \multicolumn{1}{c|}{0.3} & 0.7 & 3.2 & \multicolumn{1}{c|}{0.09} & 0.10 \\
  Cube 2 & 2.4 M & 264 K & 1.8 & \multicolumn{1}{c|}{1.1} & 2.6 & 13.1 & \multicolumn{1}{c|}{0.35} & 0.40 \\
  Cube 3 & 9.4 M & 1.0 M & 7.6 & \multicolumn{1}{c|}{4.8} & 11.0 & 53.1 & \multicolumn{1}{c|}{1.37} & 1.58 \\
  Cube 4 & 37.7 M & 4.2 M & 31.9 & \multicolumn{1}{c|}{19.4} & 44.9 & 216.0 & \multicolumn{1}{c|}{5.50} & 6.32
\end{tabular}}
\medskip
\caption{Results and performances obtained by the chunk and the DMT-based algorithms. The columns from left to right indicate: the name of the dataset ({\em Dataset}), the number of cells before and after the simplification algorithms ({\em Size}), the time ({\em Time}), expressed in seconds, needed to perform the preparation step ({\em Prep.}) and the simplification step ({\em Simpl.}), the maximum peak of memory, expressed in GB, required by the two algorithms ({\em Memory Usage}).}
\label{table:experiments}
\end{table}

The column {\em Size} in Table~\ref{table:experiments} collects the sizes (in terms of the number of simplices/columns) of the bifiltered chain complexes in input and in output.
The compression factor achieved by both algorithms varies between 9 and 23. As average, the reduced chain complex in output is approximately 13 times smaller than the original one.
Despite of the theoretical advantage that our approach
yields an optimal size in all situations, the size of the output returned by the two algorithms is nearly the same in all listed examples. A difference can be noticed just for datasets including shapes that can be considered as ``pathological''. For instance, an example of such a dataset is the conification of a dunce hat.

The column {\em Time} shows the computation times of both algorithms.
In all tested examples, the chunk algorithm takes a small fraction of time with respect to the DMT-based approach.
We also see that the creation of the boundary matrix
(preparation step) takes more time
than the chunk algorithm
(simplification step) by a factor of $2$ to $3$, but even these two steps
combined are faster than the simplification step of the DMT-based approach. It is worth
investigating whether the boundary matrix creation becomes a more
severe bottleneck for other datasets (e.g., in higher dimension),
where the DMT-based approach might have advantages by not creating
the boundary matrix explicitly.

The column {\em Memory Usage} shows the memory consumption of both approaches. The chunk algorithm does not need to encode auxiliary structures like the discrete Morse gradient stored in the DMT-based approach resulting in a slightly less amount of required space. Depending on the dataset, the chunk approach requires between 0.09 and 5.50 GB of memory and it is, on average, 1.2 times more compact than the DMT-based one.

In summary, the chunk algorithm satisfies a stronger optimality condition than \cite{Scaramuccia2018}, but still is an order of magnitude faster and uses comparable memory.

\section{Conclusion}\label{sec:concl}
We have presented a pre-processing procedure for improving the computation of multi-parameter persistent homology and we have provided theoretical and experimental evidence of its effectiveness.
In the future, we want to further improve the proposed strategy as follows.
First, we would like to develop a parallel implementation with shared memory of the chunk algorithm and compare its performances with the ones obtained by the current version.
Similarly to \cite{Scaramuccia2018}, we also want to evaluate the impact of the chunk algorithm for the computation of the persistence module and of the persistence space.

Our optimality criterion is formulated for the class of chain complexes quasi-isomorphic to the input.
It is possible to produce counterexamples showing that this statement cannot be generalized to the class of the chain complexes homology-equivalent to the input. Nevertheless, we want to further investigate the optimality properties satisfied by our algorithm and compare our algorithm with the minimal presentation algorithm for persistence modules from RIVET.




\appendix
\section{Proofs from Section~\ref{sec:correct}}
\label{sec:app_details_1}

\subparagraph{Order-preserving column additions.}
In accordance with the notations adopted in the paragraph of Section \ref{sec:correct} devoted to the analysis of order-preserving column additions, we show here the detailed proofs of the mentioned results.

\begin{proposition}
  $\bar{C}_*$ is a chain complex.
\end{proposition}

\begin{proof}
 It is enough to prove that $\bar{\partial}_{l-1} \bar{\partial}_l=0$ for $l=k, k+1, k+2$.\\
 If $l=k$, then, thanks to the fact that $C_*$ is a chain complex,
 \begin{eqnarray*}
   \bar{\partial}_{k-1} \bar{\partial}_k(c)&=& \bar{\partial}_{k-1}(\partial_k(c) + \lambda \langle c, c_2 \rangle \partial_k (c_1))\\
   &=&\partial_{k-1}(\partial_k(c)) + \lambda \langle c, c_2 \rangle \partial_{k-1}(\partial_k (c_1))\\
   &=& 0.
 \end{eqnarray*}
If $l=k+1$, then, since $C_*$ is a chain complex and $\langle c_1, c_2 \rangle=0$,
\begin{eqnarray*}
  \bar{\partial}_{k} \bar{\partial}_{k+1}(c)&=& \bar{\partial}_{k}(\partial_{k+1}(c) - \lambda \langle \partial_{k+1} (c), c_2 \rangle c_1)\\
  &=& \partial_k(\partial_{k+1}(c)) + \lambda \langle \partial_{k+1}(c), c_2 \rangle \partial_k (c_1)
  - \lambda \langle \partial_{k+1} (c), c_2 \rangle ( \partial_k c_1 + \lambda \langle c_1, c_2 \rangle \partial_k (c_1) )\\
  &=& 0.
\end{eqnarray*}
If $l=k+2$, then, since $C_*$ is a chain complex,
\begin{eqnarray*}
  \bar{\partial}_{k+1} \bar{\partial}_{k+2}(c)&=& \partial_{k+1}(\partial_{k+2}(c)) - \lambda \langle \partial_{k+1}(\partial_{k+2}(c)), c_2 \rangle c_1\\
  &=& 0.
\end{eqnarray*}
\end{proof}

\begin{lemma}
  $f_*$ and $g_*$ are chain maps.
\end{lemma}

\begin{proof}
  For proving that $f_*$ is a chain map is enough to show that $\bar{\partial}_l f_l = f_{l-1} \partial_l$ for $l=k, k+1$.\\
  If $l=k$, then
  \begin{eqnarray*}
    \bar{\partial}_k f_k(c) &=& \bar{\partial}_k(c - \lambda \langle c, c_2 \rangle c_1)\\
    &=& \partial_k(c) - \lambda \langle c, c_2 \rangle \partial_k(c_1) + \lambda \langle c-\lambda \langle c, c_2 \rangle c_1, c_2 \rangle \partial_k(c_1)\\
    &=& \partial_k(c) - \lambda \langle c, c_2 \rangle \partial_k(c_1) + \lambda \langle c, c_2 \rangle \partial_k(c_1) - \lambda^2 \langle c, c_2 \rangle \langle c_1, c_2 \rangle \partial_k(c_1)\\
    &=& \partial_k(c) = f_{k-1}(\partial_k(c)).
  \end{eqnarray*}
  If $l=k+1$, then
  \begin{eqnarray*}
    \bar{\partial}_{k+1} f_{k+1} (c) &=& \bar{\partial}_{k+1} (c) = \partial_{k+1}(c) - \lambda \langle \partial_{k+1} (c), c_2 \rangle c_1\\
    &=& f_k (\partial_{k+1}(c)).
  \end{eqnarray*}
  For proving that $g_*$ is a chain map is enough to show that $g_{l-1} \bar{\partial}_l = \partial_l g_l$ for $l=k, k+1$.\\
  If $l=k$, then
  \begin{eqnarray*}
    g_{k-1} \bar{\partial}_k(c)&=& g_{k-1} (\partial_k(c) + \lambda \langle c, c_2 \rangle \partial_k (c_1))\\
    &=& \partial_k(c) + \lambda \langle c, c_2 \rangle \partial_k (c_1) \\
    &=& \partial_k ( c + \lambda \langle c, c_2 \rangle c_1  )\\
    &=& \partial_k (g_k(c)).
  \end{eqnarray*}
  If $l=k+1$, then
  \begin{eqnarray*}
    g_k \bar{\partial}_{k+1} (c)&=& g_k ( \partial_{k+1}(c) - \lambda \langle \partial_{k+1} (c), c_2 \rangle c_1 )\\
    &=& \partial_{k+1}(c) - \lambda \langle \partial_{k+1} (c), c_2 \rangle c_1 + \lambda \langle \partial_{k+1}(c) - \lambda \langle \partial_{k+1} (c), c_2 \rangle c_1, c_2\rangle c_1\\
    &=& \partial_{k+1}(c) - \lambda \langle \partial_{k+1} (c), c_2 \rangle c_1 + \lambda \langle \partial_{k+1} (c), c_2 \rangle c_1 - \lambda^2 \langle \partial_{k+1} (c), c_2 \rangle \langle c_1, c_2 \rangle c_1\\
    &=& \partial_{k+1}(c) = \partial_{k+1}( g_{k+1} (c) ).
  \end{eqnarray*}
\end{proof}

\begin{proposition}
 $C_*$ and $\bar{C}_*$ are isomorphic.
\end{proposition}

\begin{proof}
It is enough to show that $g_k f_k = \id_{C_k}$ and $f_k g_k = \id_{\bar{C}_k}$.
We have that
\begin{eqnarray*}
  g_k f_k(c)&=&g_k(c - \lambda \langle c, c_2 \rangle c_1)\\
  &=& c - \lambda \langle c, c_2 \rangle c_1 + \lambda \langle c - \lambda \langle c, c_2 \rangle c_1, c_2 \rangle c_1\\
  &=& c - \lambda \langle c, c_2 \rangle c_1 + \lambda \langle c, c_2 \rangle c_1 - \lambda^2 \langle c, c_2 \rangle \langle c_1, c_2 \rangle c_1\\
  &=& c.
\end{eqnarray*}
Analogously,
\begin{eqnarray*}
  f_k g_k(c)&=&f_k(c + \lambda \langle c, c_2 \rangle c_1)\\
  &=& c + \lambda \langle c, c_2 \rangle c_1 - \lambda \langle c + \lambda \langle c, c_2 \rangle c_1, c_2 \rangle c_1\\
  &=& c + \lambda \langle c, c_2 \rangle c_1 - \lambda \langle c, c_2 \rangle c_1 - \lambda^2 \langle c, c_2 \rangle \langle c_1, c_2 \rangle c_1\\
  &=& c.
\end{eqnarray*}
\end{proof}

\begin{lemma}
  The function $v$ is a valid value function for $\bar{C}_*$ inducing a bifiltered chain complex $\bar{C}$.
\end{lemma}

\begin{proof}
  We have to show, that, for any $c \in \bar{C}_l$, $v(\bar{\partial}_l(c))\leq v(c)$.
  For $l\neq k, k+1$, that is trivial. Let us prove the other cases, recalling that:
  \begin{itemize}
    \item $v(c_1)\leq v(c_2)$,
    \item for any $c \in C_l$, $v(\partial_l(c))\leq v(c)$.
  \end{itemize}
  If $l=k$, $\bar{\partial}_k(c)=\partial_k(c) + \lambda \langle c, c_2 \rangle \partial_k (c_1)$.
  If $\langle c, c_2 \rangle=0$, then $v(\bar{\partial}_k(c))= v(\partial_k(c)) \leq v(c)$. Otherwise,
  \begin{itemize}
    \item $v(\partial_k(c_1))\leq v(c_1) \leq v(c_2) \leq v(c)$,
    \item $v(\partial_k(c)) \leq v(c)$.
  \end{itemize}
  Then, $v(\bar{\partial}_k(c))\leq v(c)$.\\
  If $l=k+1$, $\bar{\partial}_{k+1}(c)=\partial_{k+1}(c) - \lambda \langle \partial_{k+1} (c), c_2 \rangle c_1$.
  If $\langle \partial_{k+1} (c), c_2 \rangle=0$, then $v(\bar{\partial}_{k+1}(c))= v(\partial_{k+1}(c)) \leq v(c)$. Otherwise, $v(c_1) \leq v(c_2) \leq v(\partial_{k+1} (c)) \leq v(c)$. Then, $v(\bar{\partial}_{k+1}(c))\leq v(c)$.
\end{proof}

\begin{lemma}
  For any $p\in \R^2$, the restrictions $f^p_l:C^p_l \rightarrow \bar{C}^p_l$ and $g^p_l:\bar{C}^p_l \rightarrow C^p_l$ of the maps $f_l$ and $g_l$, respectively, are well-defined.
\end{lemma}

\begin{proof}
  For $f^p_l$, we have to prove that, for any $c \in C_l$, $v(f_l(c))\leq v(c)$.
  Let us focus on the only non-trivial case occurring for $l=k$.\\
  If $l=k$, $f_k(c)=c - \lambda \langle c, c_2 \rangle c_1$.
  If $\langle c, c_2 \rangle=0$, then $v(f_k(c))= v(c)$. Otherwise, $v(c_1) \leq v(c_2) \leq v(c)$.
  Then, $v(f_k(c))\leq v(c)$.\\
  Except for a sign, the proof for $g^p_l$ is identical.
\end{proof}

\subparagraph{Removal of local pairs.}
In accordance with the notations adopted in the paragraph of Section \ref{sec:correct} devoted to the analysis of removal of local pairs, we show here the detailed proofs of the mentioned results.

\begin{proposition}
  $\bar{C}_*$ is a chain complex.
\end{proposition}

\begin{proof}
  First of all, we have to prove that $\bar{\partial}_l$ is well-defined. In order to do that, we have to show that, for $l=k+1, k+2$, if $c\in \bar{C}_l$ then $\bar{\partial}_l(c)\in \bar{C}_{l-1}$.\\
  If $l=k+1$,
  \begin{eqnarray*}
    \langle \bar{\partial}_{k+1}(c), c_1 \rangle&=& \langle \partial_{k+1}(c), c_1 \rangle - \lambda^{-1} \langle \partial_{k+1}(c), c_1 \rangle \langle \partial_{k+1}(c_2), c_1 \rangle\\
    &=& \langle \partial_{k+1}(c), c_1 \rangle - \lambda^{-1} \lambda \langle \partial_{k+1}(c), c_1 \rangle\\
    &=& 0.
  \end{eqnarray*}
  Then, $\bar{\partial}_{k+1}(c)\in \bar{C}_{k}$.
  \\
  If $l=k+2$,
  \begin{eqnarray*}
    \langle \bar{\partial}_{k+2}(c), c_2 \rangle&=& \langle \partial_{k+2}(c), c_2 \rangle - \langle \partial_{k+2}(c), c_2 \rangle \langle c_2, c_2 \rangle\\
    &=& 0.
  \end{eqnarray*}
  Then, $\bar{\partial}_{k+2}(c)\in \bar{C}_{k+1}$.
  \\
  Second, we have to prove that $\bar{\partial}_{l} \bar{\partial}_{l+1}=0$.
  It is enough to show that for $l=k, k+1, k+2$.\\
  If $l=k$, then, thanks to the fact that $C_*$ is a chain complex,
  \begin{eqnarray*}
    \bar{\partial}_{k} \bar{\partial}_{k+1}(c)&=& \partial_{k}(\partial_{k+1}(c)) - \lambda^{-1} \langle \partial_{k+1}(c), c_1 \rangle \partial_{k}(\partial_{k+1} (c_1))\\
    &=& 0.
  \end{eqnarray*}
 If $l=k+1$, then, since $C_*$ is a chain complex,
 \begin{eqnarray*}
   \bar{\partial}_{k+1} \bar{\partial}_{k+2}(c)&=& \bar{\partial}_{k+1}(\partial_{k+2}(c) - \langle \partial_{k+2} (c), c_2 \rangle c_2)\\
   &=& \partial_{k+1}(\partial_{k+2}(c)) - \lambda^{-1} \langle \partial_{k+1}(\partial_{k+2}(c)), c_1 \rangle \partial_{k+1}(c_2) - \langle \partial_{k+2} (c), c_2 \rangle (\partial_{k+1}(c_2) - \lambda^{-1} \lambda \partial_{k+1}(c_2))\\
   &=& 0.
 \end{eqnarray*}
 If $l=k+2$, then, since $C_*$ is a chain complex,
 \begin{eqnarray*}
   \bar{\partial}_{k+2} \bar{\partial}_{k+3}(c)&=& \partial_{k+2}(\partial_{k+3}(c)) - \langle \partial_{k+2}(\partial_{k+3}(c)), c_2 \rangle c_2\\
   &=& 0.
 \end{eqnarray*}
\end{proof}

\begin{lemma}
  $r_*$ and $s_*$ are chain maps.
\end{lemma}

\begin{proof}
  For proving that $r_*$ is a chain map is enough to show that $\bar{\partial}_l r_l=r_{l-1} \partial_l$ for $l=k, k+1, k+2$.\\
  If $l=k$, then
  \begin{eqnarray*}
    \bar{\partial}_k r_k(c) &=& \partial_k(c) - \lambda^{-1} \langle c, c_1 \rangle \partial_k(\partial_{k+1}(c_2))\\
    &=& \partial_k(c) = r_{k-1}(\partial_k(c)).
  \end{eqnarray*}
  \\
  If $l=k+1$, then
  \begin{eqnarray*}
    \bar{\partial}_{k+1} r_{k+1}(c) &=& \bar{\partial}_{k+1}(c - \langle c, c_2 \rangle c_2)\\
    &=& \partial_{k+1}(c) - \lambda^{-1} \langle \partial_{k+1}(c), c_1 \rangle \partial_{k+1}(c_2) - \langle c, c_2 \rangle (\partial_{k+1}(c_2) -\lambda^{-1} \lambda \partial_{k+1}(c_2)  )\\
    &=& \partial_{k+1}(c) - \lambda^{-1} \langle \partial_{k+1}(c), c_1 \rangle \partial_{k+1}(c_2)\\
    &=& r_{k}(\partial_{k+1}(c)).
  \end{eqnarray*}
  If $l=k+2$, then
  \begin{eqnarray*}
    \bar{\partial}_{k+2} r_{k+2} (c) &=& \partial_{k+2}(c) - \langle \partial_{k+2}(c), c_2 \rangle c_2\\
    &=& r_{k+1} (\partial_{k+2}(c)).
  \end{eqnarray*}
  For proving that $s_*$ is a chain map is enough to show that $s_{l-1} \bar{\partial}_l = \partial_l s_l$ for $l=, k+1,k+2$.\\
  If $l=k+1$, then
  \begin{eqnarray*}
    s_{k} \bar{\partial}_{k+1}(c)&=& \partial_{k+1}(c) - \lambda^{-1} \langle \partial_{k+1}(c), c_1 \rangle \partial_{k+1}(c_2)\\
    &=& \partial_{k+1} (s_{k+1}(c)).
  \end{eqnarray*}
  If $l=k+2$, then
  \begin{eqnarray*}
    s_{k+1} \bar{\partial}_{k+2}(c)&=& s_{k+1} ( \partial_{k+2}(c) -  \langle \partial_{k+2} (c), c_2 \rangle c_2 )\\
    &=& \partial_{k+2}(c) - \lambda^{-1} \langle \partial_{k+1}(\partial_{k+2}(c)), c_1 \rangle c_2 - \langle \partial_{k+2}(c), c_2 \rangle (c_2 - \lambda^{-1} \lambda c_2)\\
    &=& \partial_{k+2}(c) = \partial_{k+2}( s_{k+2} (c) ).
  \end{eqnarray*}
\end{proof}

In order to prove the following result, let us define a collection of maps $\phi_l: C_l \rightarrow C_{l+1}$, for any $l\in \N$.
Given $c \in C_l$, the map $\phi_l$ is defined as:
\[\phi_l(c)=\begin{cases}
\lambda^{-1} \langle c, c_1 \rangle c_2 & \text{ if } l=k,\\
0 & \text{ otherwise.}
\end{cases}\]

\begin{proposition}
 The maps $r_*$ and $s_*$ are chain maps which are homotopy-inverse one with respect to the other.
\end{proposition}

\begin{proof}
  In order to prove the thesis, we show that, for any $l$:
  \begin{itemize}
    \item[(1)] $s_l r_l$ is chain-homotopic to $\id_{C_l}$ via $\phi_l$,
    \item[(2)] $r_l s_l = \id_{\bar{C}_l}$.
  \end{itemize}
  Let us start with proving $(1)$. We have to show that
  \[\partial_{l+1}\phi_{l} + \phi_{l-1}\partial_{l}+s_{l}r_{l}=\id_{C_l}.\]
  The only non-trivial cases occur for $l=k, k+1$.\\
  If $l=k$, we have that
  \begin{eqnarray*}
    \partial_{k+1}\phi_{k}(c) + \phi_{k-1}\partial_{k}(c) +s_{k}r_{k}(c)&=& \lambda^{-1} \langle c, c_1 \rangle \partial_{k+1}(c_2) + 0 + c - \lambda^{-1} \langle c, c_1 \rangle \partial_{k+1}(c_2)\\
    &=& c.
  \end{eqnarray*}
  If $l=k+1$, we have that
  \begin{eqnarray*}
    \partial_{k+2}\phi_{k+1}(c) + \phi_{k}\partial_{k+1}(c) + s_{k+1}r_{k+1}(c)&=& 0 +  \lambda^{-1} \langle \partial_{k+1}(c), c_1 \rangle c_2 + s_{k+1}( c - \langle c, c_2 \rangle c_2 )\\
    &=& \lambda^{-1} \langle \partial_{k+1}(c), c_1 \rangle c_2 + c - \lambda^{-1} \langle \partial_{k+1}(c), c_1 \rangle c_2 - \langle c, c_2 \rangle (c_2 - \lambda^{-1}\lambda c_2)\\
    &=& c.
  \end{eqnarray*}
  Now, let us prove $(2)$. The only non-trivial cases occur for $l=k, k+1$.\\
  If $l=k$, since $c\in \bar{C}_k$ implies $\langle c, c_1 \rangle = 0$, we have that
  \begin{eqnarray*}
    r_k s_k(c)&=&r_k(c)= c - \lambda^{-1} \langle c, c_1 \rangle \partial_{k+1}(c_2)\\
    &=& c.
  \end{eqnarray*}
  If $l=k+1$, since $c\in \bar{C}_{k+1}$ implies $\langle c, c_2 \rangle = 0$, we have that
  \begin{eqnarray*}
    r_{k+1} s_{k+1}(c)&=&r_{k+1}( c - \lambda^{-1} \langle \partial_{k+1}(c), c_1 \rangle c_2 )\\
    &=& c - \langle c, c_2 \rangle c_2 - \lambda^{-1} \langle \partial_{k+1}(c), c_1 \rangle (c_2 - c_2)\\
    &=& c.
  \end{eqnarray*}
\end{proof}

\begin{lemma}
  The function $v$ is a valid value function for $\bar{C}_*$ inducing a bifiltered chain complex $\bar{C}$.
\end{lemma}

\begin{proof}
  We have to show, that, for any $c \in \bar{C}_l$, $v(\bar{\partial}_l(c))\leq v(c)$.
  For $l\neq k+1, k+2$, that is trivial. Let us prove the other cases, recalling that:
  \begin{itemize}
    \item $v(c_1)= v(c_2)$,
    \item for any $c \in C_l$, $v(\partial_l(c))\leq v(c)$.
  \end{itemize}
  If $l=k+1$, $\bar{\partial}_{k+1}(c)=\partial_{k+1}(c) - \lambda^{-1} \langle \partial_{k+1}(c), c_1 \rangle \partial_{k+1}(c_2)$.
  If $\langle \partial_{k+1}(c), c_1 \rangle=0$, then $v(\bar{\partial}_{k+1}(c))= v(\partial_{k+1}(c)) \leq v(c)$. Otherwise, \item $v(\partial_{k+1}(c_2))\leq v(c_2) = v(c_1) \leq v(\partial_{k+1}(c)) \leq v(c)$.
  Then, $v(\bar{\partial}_{k+1}(c))\leq v(c)$.\\
  If $l=k+2$, $\bar{\partial}_{k+2}(c)=\partial_{k+2}(c) - \langle \partial_{k+2} (c), c_2 \rangle c_2$.
  If $\langle \partial_{k+2} (c), c_2 \rangle=0$, then $v(\bar{\partial}_{k+2}(c))= v(\partial_{k+2}(c)) \leq v(c)$.
  Otherwise, $v(c_2) \leq v(\partial_{k+2} (c)) \leq v(c)$.
  Then, $v(\bar{\partial}_{k+2}(c))\leq v(c)$.
\end{proof}

\begin{lemma}
  For any $p\in \R^2$, the restrictions $r^p_l:C^p_l \rightarrow \bar{C}^p_l$, $s^p_l:\bar{C}^p_l \rightarrow C^p_l$, and $\phi^p_l:C^p_l \rightarrow C^p_{l+1}$ of the maps $r_l$, $s_l$, and $\phi_l$, respectively, are well-defined.
\end{lemma}

\begin{proof}
  For $r^p_l$, we have to prove that, for any $c \in C_l$, $v(r_l(c))\leq v(c)$.
  Let us focus on the only non-trivial cases occurring for $l=k,k+1$.\\
  If $l=k$, $r_k(c)= c - \lambda^{-1} \langle c, c_1 \rangle \partial_{k+1}(c_2)$.
  If $\langle c, c_1 \rangle=0$, then $v(r_k(c))= v(c)$. Otherwise, $v(\partial_{k+1}(c_2)) \leq v(c_2) = v(c_1) \leq v(c)$.
  Then, $v(r_k(c))\leq v(c)$.\\
  If $l=k+1$, $r_{k+1}(c)=c - \langle c, c_2 \rangle c_2$.
  If $\langle c, c_2 \rangle=0$, then $v(r_{k+1}(c))= v(c)$. Otherwise, $v(c_2) \leq v(c)$.
  Then, $v(r_{k+1}(c))\leq v(c)$.\\
  For $s^p_l$, we have to prove that, for any $c \in \bar{C}_l$, $v(s_l(c))\leq v(c)$.
  Let us focus on the only non-trivial case occurring for $l=k+1$.\\
  If $l=k+1$,
  $s_{k+1}(c)=c - \lambda^{-1} \langle \partial_{k+1}(c), c_1 \rangle c_2$.
  If $\langle \partial_{k+1}(c), c_1 \rangle=0$, then $v(s_{k+1}(c))= v(c)$. Otherwise, $v(c_2) = v(c_1) \leq v(\partial_{k+1}(c)) \leq v(c)$.
  Then, $v(s_{k+1}(c))\leq v(c)$.\\
  For $\phi^p_l$, we have to prove that, for any $c \in C_l$, $v(\phi_l(c))\leq v(c)$.
  Let us focus on the only non-trivial case occurring for $l=k$ and for $c \in C_k$ such that $\langle c, c_1 \rangle \neq 0$.\\
  If this is the case, then $v(c_2)=v(c_1)\leq v(c)$. So,
  $v(\phi_k(c))=v(\lambda^{-1}\langle c, c_1 \rangle c_2)=v(c_2)\leq v(c)$.
\end{proof}

\section{Proofs from Section~\ref{sec:optimal}}
\label{sec:app_details_2}
In accordance with the notations adopted in Section \ref{sec:optimal}, we show here the detailed proofs of the mentioned results.

Let $A_*, B_*$ be two chain complexes such that $A_*$ is a subcomplex of $B_*$ and let $\eta_k: H_k(A_*)\rightarrow H_k(B_*)$ be the map in homology induced by this inclusion.
Let us consider the boundary matrix of $B_*$ in which the generators of $A_*$ are placed on the left with respect to the generators of $B_*$ not belonging to $A_*$. By applying the standard reduction algorithm for computing persistent homology \cite{Edelsbrunner2010}, we obtain a reduced matrix $R$.
To avoid a clash in terminology with previous parts of the paper,
we refer to the entry of lowest position in the matrix among the non-null ones in column $c'$ as the \emph{lowest index} of $c'$.
Let us denote the columns of $R$ corresponding to the generators of $B_*$ not belonging to $A_*$ as the columns of $B_*\setminus A_*$ and the others as the columns of $A_*$.

\begin{lemma}\label{lemma:matrixReduction}
  We have that:
  \begin{itemize}
    \item the number of the $k$-columns of $B_*\setminus A_*$ having as lowest index a column of $A_*$ is equal to $\dim\,\ker \, \eta^p_{k-1}$,
    \item the number of the $k$-columns of $B_*\setminus A_*$ with null boundary for which there is no column in $R$ having them as lowest index is equal to $\dim\,\coker \, \eta^p_k$.
  \end{itemize}
\end{lemma}

\begin{proof}
We can partition the $k$-columns of $B_*\setminus A_*$ in three distinct classes.
The first class includes the $k$-columns of $B_*\setminus A_*$ having as lowest index a column of $A_*$. According to the standard reduction algorithm, these columns are in one-to-one correspondence with the non-null homology classes of $H_{k-1}(A_*)$ that become trivial in $H_{k-1}(B_*)$. So, their number coincides with the dimension of $\ker \, \eta^p_{k-1}$.
The second class includes the $k$-columns of $B_*\setminus A_*$ with null boundary for which there is no column in $R$ having them as lowest index.
According to the standard reduction algorithm, these columns are in one-to-one correspondence with the homology classes that are born in $H_{k}(B_*)$. So, their number coincides with the dimension of $\coker \, \eta^p_{k}$.
The third class of $k$-columns of $B_*\setminus A_*$ are the ones with lowest index in $B_*\setminus A_*$ or the columns with null boundary which are the lowest index of a column of $B_*\setminus A_*$.
\end{proof}

\begin{theorem}\label{thm:optimality2}
  Let $\bar{C}$ be the bifiltered chain complex obtained by applying the proposed reduction algorithm to the bifiltered chain complex $C$. The number $\gamma^p_k(\bar{C})$ of $k$-generators added at value $p$ in $\bar{C}$ is equal to $\delta^p_k(C)$.
\end{theorem}

\begin{proof}
Let us apply Lemma \ref{lemma:matrixReduction} to the case $A_*=C^{<p}_*$ and $B_*=C^p_*$
(considered as unfiltered chain complexes). Note that the local reduction phase
at value $p$ in the chunk algorithm corresponds to a partial reduction
of the matrix of $C^p_*$, because all column additions are left-to-right,
and it is known that the lowest indices of the reduced matrix are invariant
under left-to-right column additions. At the end of the local reduction phase,
each column is labeled as local (positive or negative), or global.
A local column is paired with another local column of value $p$,
and this pairing will not change when completing the partial reduction
to a completely reduced matrix for $C^p_*$.
The reason is simply that all columns
in $C^{<p}_*$ have a lowest index with smaller index
than any column of value $p$, since the boundary matrix is upper triangular.
By the above reasoning, the local columns will neither contribute to
kernel nor cokernel of $\eta^p_{k-1}$.

On the other hand, consider a global column with value $p$. Its lowest index
corresponds to a column in $C^{<p}_*$ after the local reduction phase.
Since the lowest index can only decrease when further reducing the matrix,
it follows that in the reduced matrix, the global column has
still a lowest index in $C^{<p}_*$, in which case it contributes to $\ker\,\eta^p_{k-1}$, or its boundary becomes null. In the latter case, it cannot be paired with
another column at value $p$ since then, this pairing would have been
identified already at the local reduction phase. Hence, the global column
contributes to $\coker \, \eta^p_{k}$.
After all, each global column contributes to the kernel or the cokernel,
and therefore, the number of global columns is equal to $\delta^p_k(C)$.
Moreover, the number of global columns at value $p$ is exactly the number
of generators added at value $p$ to the output complex $\bar{C}$.
\end{proof}

Given an arbitrary bifiltered chain complex $D$ and $p=(p_x, p_y)\in \R^2$, let us consider the chain complex $D^{<p}_*$ defined as $\sum_{q<p} D^q_*$.
Since we are considering just finitely generated chain complexes, there exists a sufficiently small $\epsilon$ for which, by defining $p_1=(p_x-\epsilon, p_y)$ and $p_2=(p_x, p_y-\epsilon)$, we have that, for any $p_1\leq q \leq p$, $D^q_*=D^{p_1}_*$, and, for any $p_2\leq q \leq p$, $D^q_*=D^{p_2}_*$. Then, $D^{<p}_*$ can be expressed as the finite sum $D^{p_1}_* + D^{p_2}_*$.
Moreover, defining $p_0$ as the value $(p_x-\epsilon, p_y-\epsilon)$, the 1-criticality of the bifiltration $D$ ensures that $D^{p_0}_*=D^{p_1}_* \cap D^{p_2}_*$.
For any arbitrary bifiltered chain complex $D$, let $\eta^p_k$ denote the homology map in dimension $k$ induced by the inclusion of $D^{<p}_*$ into $D^p_*$.
Since each $k$-generator newly introduced in $D^p_*$ can act on the homology by just increasing the dimension of $\ker \, \eta^p_{k-1}$, or alternatively of $\coker \, \eta^p_k$, at most by 1, we have that
\begin{equation}\label{eq:ker-coker}
\delta^p_k(D)=\dim\,\ker \, \eta^p_{k-1} + \dim\,\coker \, \eta^p_k \leq \dim\,D^p_k - \dim\,D^{<p}_k=\gamma^p_k(D).
\end{equation}

\begin{theorem}\label{lemma:leq2}
Let $D$ be a bifiltered chain complex quasi-isomorphic to $C$ (through the chain maps $f^p_*: C^p_* \rightarrow D^p_*$). Then, $\delta^p_k(C) \leq \gamma^p_k(D)$.
\end{theorem}

\begin{proof}
Let $f^p_*: C^p_* \rightarrow D^p_*$ be the collection of chain maps ensuring the quasi-isomorphism between $C$ and $D$.
Thanks to their commutativity with the inclusion maps, we have that the chain map $f^{<p}_*: C^{<p}_* \rightarrow D^{<p}_*$ is well-defined.
This ensures for any $p \in \R^2$, the construction of the following commutative diagram
\begin{diagram}
0 &\rTo_{\qquad} & C^{p_0}_k         &\rTo_{\qquad}   & C^{p_1}_k \oplus C^{p_2}_k &\rTo_{\qquad}   & C^{<p}_k &\rTo_{\qquad}   & 0 \\
& & \dTo_{f^{p_0}_k \,}  &           &\dTo_{(f^{p_1}_k, f^{p_2}_k) \,} &           &\dTo_{f^{<p}_k \,} &           & \\
0 &\rTo_{\qquad} & D^{p_0}_k         &\rTo_{\qquad}   & D^{p_1}_k \oplus D^{p_2}_k &\rTo_{\qquad}   & D^{<p}_k &\rTo_{\qquad}   & 0
\end{diagram}
in which each row is exact.
Thanks to the quasi-isomorphism established by $f$ and to the Mayer-Vietoris sequence, we can derive from the previous one the following commutative diagram
\begin{diagram}
H_k(C^{p_0}_*)         &\rTo_{\qquad}   &H_k(C^{p_1}_*) \oplus H_k(C^{p_2}_*) &\rTo_{\qquad}   &H_{k}(C^{<p}_*) &\rTo_{\qquad}   &H_{k-1}(C^{p_0}_*) &\rTo_{\qquad}   &H_{k-1}(C^{p_1}_*) \oplus H_{k-1}(C^{p_2}_*) \\
\dTo_{\psi^{p_0}_k \,}  &           &\dTo_{(\psi^{p_1}_k, \psi^{p_2}_k) \,} &           &\dTo_{\psi^{<p}_k \,} &           &\dTo_{\psi^{p_0}_{k-1} \,} &           &\dTo_{(\psi^{p_1}_{k-1}, \psi^{p_2}_{k-1}) \,}\\
H_k(D^{p_0}_*)         &\rTo_{\qquad}   &H_k(D^{p_1}_*) \oplus H_k(D^{p_2}_*) &\rTo_{\qquad}   &H_{k}(D^{<p}_*) &\rTo_{\qquad}   &H_{k-1}(D^{p_0}_*) &\rTo_{\qquad}   &H_{k-1}(D^{p_1}_*) \oplus H_{k-1}(D^{p_2}_*) \\
\end{diagram}
in which each row is exact and the horizontal and the vertical maps are induced by the corresponding ones between chain complexes.
Since all the other four vertical maps are isomorphisms, the $5$-lemma ensures that map $\psi^{<p}_k$ is also an isomorphism. So, we have the following commutative diagram
\begin{diagram}
H_k(C^{<p}_*)         &\rTo^{\quad \iota^p_k \quad}   &H_k(C^p_*)\\
\dTo_{\psi^{<p}_k \,}  &           &\dTo_{\psi^p_k \,}\\
H_k(D^{<p}_*)         &\rTo^{\quad \eta^p_k \quad}   &H_k(D^p_*)
\end{diagram}
in which vertical maps are isomorphisms.
Then, for any $k$, $\dim\,\ker \, \eta^p_{k}=\dim\,\ker \, \iota^p_{k} $ and $\dim\,\coker \, \eta^p_k=\dim\,\coker \, \iota^p_k$.
So, by applying Equation (\ref{eq:ker-coker}), we have that
\[ \delta^p_k(C) = \dim\,\ker \, \iota^p_{k-1} + \dim\,\coker \, \iota^p_k = \dim\,\ker \, \eta^p_{k-1} + \dim\,\coker \, \eta^p_k \leq \dim\,D^p_k - \dim\,D^{<p}_k = \gamma^p_k(D). \]
\end{proof}

\end{document}